\documentclass[11pt]{amsart}
\usepackage{graphicx}
\usepackage{graphicx}
\usepackage{color}
\usepackage{times}
\usepackage{cite}
\usepackage{enumerate,latexsym}
\usepackage{latexsym}
\usepackage{amsmath,amssymb}
\usepackage{graphicx}
\usepackage{amsthm}
\usepackage{verbatim}
\newtheorem{thm}{Theorem}[section]
\newtheorem{cor}[thm]{Corollary}

\newtheorem{assumption}[thm]{Assumption}
\newtheorem{lem}[thm]{Lemma}
\newtheorem{prop}[thm]{Proposition}
\theoremstyle{definition}
\newtheorem{defn}[thm]{Definition}
\theoremstyle{remark}
\newtheorem{rem}[thm]{Remark}

\numberwithin{equation}{section}

\newcommand{\set}[1]{\left\{#1\right\}}
\newcommand{\Real}{\mathbb R}

\newcommand{\func}[1]{\ensuremath{\mathop{\mathrm{#1}}} }

\newcommand{\dist}[0]{\mathrm{dist}}

\newcommand{\spt}[0]{\func{spt}}

\newcommand{\sing}[0]{\func{sing}}

\newcommand{\reg}[0]{\func{reg}}

\newcommand{\vV}[0]{\mathbf{v}}

\newcommand{\LL}[0]{\mathcal{L}}

\newcommand{\R}{\mathbb{R}}

\newcommand{\wt}{\widetilde}

\newtheoremstyle{TheoremNum}
        {\topsep}{\topsep}              
        {\itshape}                      
        {}                              
        {\bfseries}                     
        {.}                             
        { }                             
        {\thmname{#1}\thmnote{ \bfseries #3}}
    \theoremstyle{TheoremNum}
    \newtheorem{theoremn}{Theorem}

\begin{document}
\title[Topological Type of Limit Laminations]{Topological type of Limit Laminations of Embedded Minimal Disks}
\author{Jacob Bernstein \and Giuseppe Tinaglia}
\thanks{The first author was partially supported by the EPSRC Programme Grant entitled ``Singularities of Geometric Partial Differential Equations'' grant number EP/K00865X/1. The second author was partially supported by   EPSRC grant no. EP/L003163/1}

\begin{abstract}
 We consider two natural classes of minimal laminations in three-manifolds.  
Both classes may be thought of as limits -- in different senses -- of embedded 
minimal disks.  In both cases, we prove that, under a natural geometric 
assumption on the three-manifold, the leaves of these laminations are 
topologically either disks, annuli or M\"obius bands.  This answers a question 
posed by Hoffman and White.
\end{abstract}
\maketitle
\section{Introduction}
Let $\Omega$ be a a fixed Riemannian three-manifold.  Suppose that 
$\Omega_i\subset \Omega$ is an increasing sequence of open sets with 
$\Omega=\bigcup \Omega_i$ and that  $\Sigma_i$ are minimal surfaces properly 
embedded in $\Omega_i$. We say that the curvatures of the $\Sigma_i$ blow up at 
a point $p\in\Omega$ if there exists a sequence of points $p_i\in \Sigma_i$ 
converging to $p$  such that  $|A_{\Sigma_i}|(p_i)$ becomes arbitrarily large; 
where $|A_{\Sigma}|$ denotes the norm of the second fundamental form of 
$\Sigma$. We call such $p$ a blow-up point and observe that the set, $K$, of 
blow-up points is closed in $\Omega$.  The points of $K$ are precisely the 
obstruction to the sequence smoothly subconverging. Indeed, up to passing to a 
subsequence,  the $\Sigma_i\backslash  K$ converge  on compact subsets of 
$\Omega\backslash K$ to a smooth proper minimal lamination $\mathcal L$ in 
$\Omega\backslash K$ -- see~\cite[Appendix B]{cm24}. Recall, a 
lamination is a foliation that need not fill space.  We call the quadruple $(\Omega, K, \LL, \mathcal{S})$ a \emph{minimal surface sequence}.  A natural question is:

\begin{quote}
In a minimal surface sequence, what singular sets, $K$, and limit laminations, $\mathcal L$, can arise?
\end{quote}

Work of Anderson~\cite{an1} and White~\cite{wh5}, answers this question when the 
total extrinsic curvatures of the surfaces $\Sigma_i$ -- i.e., 
$\int_{\Sigma_i}|A_{\Sigma_i}|^2$ -- are uniformly bounded. In this case, $K$ is 
finite and $\mathcal L$ extends smoothly across $K$. 
In general, without such a strong assumption on the geometry of the surfaces in the sequence one does not expect such a complete answer.   

Remarkably,   when the $\Sigma_i$ are assumed only to be disks, an elegant story also emerges -- we call such minimal surface sequences, \emph{minimal disk sequences}.  In a series of papers~\cite{cm21, cm22, cm24, cm23}, Colding and Minicozzi extensively studied these sequences and proved deep structural results about both their singular sets, $K$, and limit laminations, $\LL$.  Specifically, they showed that $K$ must be contained in a Lipschitz curve and that for any point $p\in K$ there exists a leaf of $\mathcal L$  that extends smoothly across $p$.  When $\Omega=\Real^3$, Colding and Minicozzi further showed that either $K=\emptyset$ or $\mathcal{L}$ is a foliation of $\Real^3\backslash K$ by parallel planes and that $K$ consists of a connected Lipschitz curve which meets the leaves of $\LL$ transversely.  Using this result, Meeks and Rosenberg showed in~\cite{mr8} that the helicoid is the unique non-flat properly embedded minimal disk in $\Real^3$ -- see also~\cite{Bernstein2011a}. This uniqueness 
was then used by Meeks~\cite{me30} to prove that if $\Omega=\Real^3$ and $K\neq \emptyset$, then $K$ is a line orthogonal to the leaves of $\LL$. This is precisely the limit of a sequence of rescalings of a helicoid.

An example constructed by Colding and Minicozzi in~\cite{cm28}, illustrates how 
for general regions such a simple and complete description of the limit 
lamination does not hold.  Specifically, they constructed a sequence of properly 
embedded minimal disks in the unit ball $B_1$ of $\Real^3$ which has $K=\set{0}$ 
and whose limit lamination consists of three leaves -- two are non-proper and 
spiralling into the third which is the punctured unit disk in the $x_3$-plane.  
Inspired by this, a plethora of examples have now been constructed which show 
that the singular set $K$ can consist of any closed subset of a line, 
see~\cite{Dean2006, Khan, Kleene, HoffmanWhite1}.  Likewise,  Meeks and Weber 
\cite{mwe1} have given examples where $K$ is curved. Strikingly, Hoffman and 
White~\cite{hofw2} have also constructed  minimal disk sequences in which $K=\emptyset$ and 
the limit lamination $\LL$ has a leaf which is a proper annulus in $\Omega$.  In 
all examples the 
leaves are either topologically disks or annuli. This motivated Hoffman and 
White to ask in~\cite{hofw2}:

\begin{quote}
Can a surface of positive genus occur as a leaf of the lamination $\LL$  of a 
minimal disk sequence? A planar domain with more than two holes?
\end{quote}

In this paper we show that, under natural geometric assumptions on $\Omega$, 
the answer to both questions is no. 
That is, the leaves of $\mathcal L$ must be either disks or annuli.
\begin{thm} \label{BeautifulThm}
 Let $\Omega$ be the interior of a compact oriented three-manifold $N=\overline{\Omega}$ with mean-convex boundary. If $\Omega$ contains no closed minimal surfaces and 
$(\Omega, K, \mathcal{L}, \mathcal{S})$ is a minimal disk sequence, then the 
leaves of $\LL$ are either disks or annuli.  Furthermore, if $L$ is a leaf of 
$\mathcal L$ with the property that  $\overline{L}$ -- the closure in $\Omega$  of 
$L$ -- is a properly embedded minimal surface, then $\overline{L}$ is either a disk 
or it is an annulus which is disjoint from $K$.
\end{thm}

The theorem is proved by realizing the disks in the sequence $\mathcal{S}$ as 
effective universal covers of the leaves of $\LL$.  Geometric considerations -- 
specifically the fact that the disks are embedded -- strongly restrict these 
covers and this  restricts the topology of the leaves as claimed.  Our proof 
uses relatively elementary topological and geometric properties of embedded 
minimal disks -- in particular, we do not directly use the deep results of 
\cite{cm21, cm22, cm23, cm24}. 

Simple examples show that it is possible for the leaves of $\LL$ to be planar 
domains with more than two holes when  the boundary is not assumed to be
mean-convex.  Nevertheless, the methods of the present paper continue to 
show that the leaves are genus 
zero.  However,  as treating this more general case introduces several 
technical 
points,  we do not persue it. 

We further remark that it is unclear whether the condition that $\Omega$ admits 
no closed minimal surface is necessary.  To better understand this question we 
consider also the topological type of leaves of a different -- though related -- 
class of minimal laminations. 
Specifically, we say that a quadruple $(\Omega, K, \LL, \LL_0)$ is  a 
\emph{minimal disk closure} if  $\LL_0$ is a minimal lamination in $\Omega$ all 
of whose leaves are disks, $K$ is a relatively closed subset of $\Omega$ and 
$\LL=\overline{\LL}_0$, the closure of $\LL_0$ in $\Omega\backslash K$, is a proper 
minimal lamination in $\Omega\backslash K$.
\begin{thm} \label{BeautifulThm2}
Let $\Omega$ be the interior of a compact 
oriented three-manifold with boundary 
$N=\overline{\Omega}$ which has mean-convex boundary. If $\Omega$ contains no closed minimal surfaces and  
$(\Omega, K, \mathcal{L}, \mathcal{L}_0)$ is a minimal disk closure, then the 
leaves of $\LL$ are either topologically open subsets $\mathbb{S}^2$ or open subsets $\mathbb{RP}^2$.  
\end{thm}
Notice that the conclusions of Theorem \ref{BeautifulThm2} are weaker than those of Theorem \ref{BeautifulThm} in two ways.  First, we cannot rule out that the presence of many punctures. Secondly, we can no longer rule out the existence of one-sided leaves of $\LL$.  In Appendix \ref{ExampleApp}, we give an example to show this weakening is unavoidable -- that is, we construct a minimal disk closures whose leaves have multiple punctures and a minimal disk closure $(\Omega, \emptyset, \LL, \LL_0)$ where $\Omega$ is a solid torus and $\LL$ contains a M\"obius band as a leaf.  We also give an example in Appendix \ref{ExampleApp} that shows that Theorem \ref{BeautifulThm2} is  sharp -- i.e., we construct a minimal disk closure $(\Omega, \emptyset, \LL, \LL_0)$ for which one leaf of $\LL$ is a torus.  

\section{Notation}
Fix a smooth oriented Riemannian three-manifold $(\Omega,g)$.  We denote by 
$\dist^\Omega$ the distance function on $\Omega$ and by $\exp^\Omega$ the 
exponential map. Hence, 
\[
 \exp_p^\Omega:B_r\to \mathcal{B}_r(p)
\]
where $B_r$ is the usual euclidean ball in $\Real^3$ centered at the origin and $\mathcal{B}_r(p)$ is the geodesic ball in $\Omega$.
A subset  $\Sigma\subset \Omega$ is an \emph{embedded smooth surface} if for 
each point $p\in \Sigma$ there is a radius $r_p>0$ and diffeomorphism 
$\phi_p:\mathcal{B}_{r_p}(p)\to B_1$ so that $\mathbb{D}_1=\phi_p(\Sigma\cap 
\mathcal{B}_{r_p}(p))$.  Here, $\mathbb{D}_1=B_1\cap \set{x_3=0}\subset B_1$ is 
the unit disk. Such $\Sigma$ is said to be \emph{proper} in $\Omega$ if it is a 
closed subset of $\Omega$, that is, $\overline{\Sigma}=\Sigma$.


For an embedded smooth surface, $\Sigma$, we write
\[
\exp^\perp\colon N\Sigma   \to \Omega
\]
for the normal exponential map where here $N\Sigma$ is the normal bundle.
If $N\Sigma$ is trivial then we say that $\Sigma$ is \emph{two-sided}, otherwise we say that $\Sigma$ is \emph{one-sided}.  As $\Omega$ is oriented, $\Sigma$ is two-sided if and only if $\Sigma$ is orientable.
For a subset $U\subset N\Sigma$ set
\begin{equation*}
 \mathcal{N}_U(\Sigma):=\exp^\perp(U).
\end{equation*}
The set $\mathcal{N}_U(\Sigma)$ is \emph{regular} if there is an open set $V$ with $U\subset V$ so that $\exp^\perp:V\to \mathcal{N}_V(\Sigma)$ is a diffeomorphism.
If $\mathcal{N}_U(\Sigma)$ is regular, then the map $\Pi_\Sigma:\mathcal{N}_U(\Sigma)\to \Sigma$ given by nearest point projection is smooth and for any $(q,\vV)\in T\mathcal{N}_U(\Sigma)$, there is a natural splitting  $\vV=\vV^\perp+\vV^\top$, where $\vV^\perp$ is orthogonal to $\vV^\top$ and $\vV^\perp$ is tangent to the fibers of $\Pi_\Sigma$.  We say such $\vV$ is \emph{$\delta$-parallel to $\Sigma$} if 
\[
 |\vV^\perp|\leq \delta |\vV| \quad \mbox{and} \quad 
\frac{1}{1+\delta}|\vV^\top|\leq  |d(\Pi_\Sigma)_q(\vV)|\leq (1+\delta) |\vV^\top|.
\]
Given $\epsilon>0$ we set $U_\epsilon=\set{(p, \vV)\in N\Sigma: |\vV|<\epsilon}$ and define $\mathcal{N}_\epsilon(\Sigma)$, the \emph{$\epsilon$-neighborhood of $\Sigma$}, to be $\mathcal{N}_{U_\epsilon}(\Sigma)$. 
If $\Sigma$ is an embedded smooth surface and $\Sigma_0$ is a pre-compact 
subset, then there is an $\epsilon>0$ so that $\mathcal{N}_\epsilon(\Sigma_0)$ 
is regular. 

Given a fixed embedded surface $\Sigma$ and $\delta\geq 0$ 
we say that another embedded smooth surface $\Gamma$ is a \emph{smooth 
$\delta$-graph} over $\Sigma$ if there is an $\epsilon>0$ so that the following 
holds:
\begin{enumerate}
\item $\mathcal{N}_\epsilon(\Sigma)$ is a regular $\epsilon$-neighborhood of $\Sigma$;
\item Either $\Gamma$ is a proper subset of $\mathcal{N}_\epsilon(\Sigma)$ or 
$\Gamma$ is a proper subset of $\mathcal{N}_\epsilon(\Sigma)\backslash \Sigma$;
\item  Each $(q, \vV)\in T\Gamma$ is $\delta$-parallel to $\Sigma$.
\end{enumerate}
We say that a {smooth $\delta$-graph} over 
$\Sigma$, $\Gamma$, is a \emph{smooth $\delta$-cover} of 
$\Sigma$ if it is connected and 
$$\Pi_\Sigma(\Gamma)=\Sigma.$$

Let $\gamma\colon [0,1]\to \Sigma$ be a $C^1$ curve in $\Sigma$.  We will also 
denote the image of such $\gamma$ by $\gamma$.  We say that a curve $\widetilde{\gamma}\colon [0,1] \to \mathcal 
N_\delta(\gamma)$ is a \emph{$\delta$-lift} of $\gamma$ if $\mathcal 
N_\delta(\gamma)$ is regular, $\Pi_\Sigma\circ \widetilde{\gamma} = \gamma$ and 
for each $t\in[0,1]$, $(\widetilde{\gamma}(t), \widetilde{\gamma}'(t))$ is 
$\delta$-parallel to $\Sigma$.  This definition extends to piece-wise $C^1$ 
curves in an obvious manner.

\section{Minimal Laminations}\label{MinLamSec}
We recall some facts about laminations. 
\begin{defn}
A subset $\mathcal{L}\subset \Omega$ is a \emph{smooth lamination} if for each 
$p\in \mathcal{L}$, there is a radius $r_p>0$, maps $\phi_p, \psi_p: 
\mathcal{B}_{r_p}(p)\to B_1\subset \Real^3$ and a closed set $0\in T_p\subset 
(-1,1)$ so that:
 \begin{enumerate}
  \item \label{lam1}$\phi_p(p)=\psi_p(p)=0$;
  \item \label{lam2}$\phi_p$ is a smooth diffeomorphism and $\mathbb{D}_1\subset \phi_p(\mathcal{L}\cap \mathcal B_{r_p}(p))$;
  \item \label{lam3} $\psi_p$ is a Lipschitz diffeomorphism and $B_1\cap \set{x_3=t}_{t\in T_p}=\psi_p(\mathcal{L}\cap \mathcal B_{r_p}(p))$;
  \item \label{lam4}$\psi_p^{-1}(\mathbb{D}_1)=\phi_p^{-1} (\mathbb{D}_1)$.
 \end{enumerate}
 We refer to maps $\phi_p$ satisfying \eqref{lam1} and \eqref{lam2} as \emph{smoothing maps} of $\LL$ and to maps $\psi_p$ satisfying \eqref{lam1} and \eqref{lam3} as \emph{straightening maps} of $\LL$. 
\end{defn}

A smooth lamination $\mathcal L\subset \Omega$ is \emph{proper} in $\Omega$ if it is closed -- i.e. $\overline{\mathcal{L}}=\mathcal{L}$.
Any embedded smooth surface is a smooth lamination, which is 
proper if and only if the surface is.

\begin{defn}
 Let $\mathcal{L}\subset \Omega$ be a non-empty smooth lamination.  A subset 
$L\subset \mathcal{L}$ is a \emph{leaf of $\mathcal{L}$} if it is a connected, 
embedded surface and for any $p\in L$, there is an $r_p>0$ and a smoothing map $\phi_p$ so that $\mathbb{D}_1 =\phi_p(L\cap 
\mathcal{B}_{r_p}(p))$. 
 For each $p\in \mathcal{L}$, let $L_p$  be the unique leaf of $\mathcal L$ 
containing $p$.
\end{defn}

\begin{defn}
 A smooth lamination $\mathcal{L}$  is a \emph{minimal lamination} if each leaf is minimal.
\end{defn}
A sequence $\set{\Omega_i}$ of open subsets of $\Omega$ \emph{exhausts} $\Omega$ if $\Omega_{i}\subset \Omega_{i+1}$ and  $\Omega=\bigcup_{i=1}^\infty \Omega_i$.
\begin{defn}
Suppose the sequence $\set{\Omega_i}$ exhausts $\Omega$ and that $\mathcal{L}_i$ 
are  smooth proper laminations  in $\Omega_i$.   For any $0<\alpha<1$, the 
$\mathcal{L}_i$ converge in $C^{\infty, \alpha}_{loc}(\Omega)$ to 
$\mathcal{L}$, 
a proper smooth lamination in $\Omega$, provided:
\begin{enumerate}
 \item  The sets $\mathcal{L}_i$ converge to $\mathcal{L}$ in pointed Gromov-Hausdorff distance;
 \item  Smoothing maps of the $\mathcal{L}_i$ converge in $C^\infty$ to 
smoothing maps of $\mathcal{L}$.  That is, for each $p\in \mathcal{L}$ there 
is an $r_p>0$ and an $i_p>0$ so that: for $i>i_p$, 
$\mathcal{B}_{r_p}(p)\subset \Omega_i$ and for all 
$p_i\in\mathcal{B}_{\frac{1}{4}r_p}(p)\cap \mathcal{L}_i$ converging to $p$,  there are  $r_{p_i}\geq r_p$ and smoothing maps  
$\phi_i:\mathcal{B}_{\frac{1}{2} r_p}(p_i)\to B_1$ of the $\LL_i$ 
converging in $C^\infty_{loc}(\mathcal B_{\frac{1}{4}r_p}(p))$ to a 
smoothing map $\phi_p:\mathcal B_{\frac{1}{4}r_p}(p)\to B_1$ of $\LL$.
\item Straightening maps of the $\mathcal{L}_i$ converge in $C^\alpha$ to 
straightening maps of $\mathcal{L}$.  That is, for each $p\in \mathcal{L}$ 
there 
is an $r_p>0$ and an $i_p>0$ so that: for $i>i_p$, 
$\mathcal{B}_{r_p}(p)\subset \Omega_i$ and for all 
$p_i\in\mathcal{B}_{\frac{1}{4}r_p}(p)\cap \mathcal{L}_i$ converging to $p$, there are  $r_{p_i}\geq r_p$ and straightening maps  
$\psi_i:\mathcal{B}_{\frac{1}{2} r_p}(p_i)\to B_1$ of the $\LL_i$ 
converging in $C^\alpha_{loc}(\mathcal B_{\frac{1}{4}r_p}(p))$ to a 
straightening map $\psi_p:\mathcal B_{\frac{1}{4}r_p}(p)\to B_1$ of $\LL$.
\end{enumerate}
\end{defn}

The following is the natural compactness result for sequences of properly 
embedded minimal surfaces with uniformly bounded second fundamental form -- 
see~\cite[Appendix B]{cm24} for a proof.  

\begin{thm} \label{SmoothCompactThm}
 Suppose that $\set{\Omega_i}$ exhausts $\Omega$ and that $\Sigma_i$ are 
properly embedded smooth minimal surfaces in $\Omega_i$. 
If for each compact subset $U$ of $\Omega$   there is a constant $C(U)<\infty$ so that when $U\subset \Omega_i$ 
\begin{equation*}
 \sup_{U\cap \Sigma_i} |A_{\Sigma_i}|\leq C(U),
\end{equation*}
then, for any $0<\alpha<1$, up to passing to a subsequence, the $\Sigma_i$ 
converge in $C^{\infty,\alpha}_{loc}(\Omega)$ to  a smooth 
proper minimal lamination $\mathcal{L}$ in $\Omega$.
\end{thm}
\begin{rem} \label{straighteningrem}
 While the straightening maps converge in $C^\alpha$, their Lipschitz norms are uniformly bounded  on compact subsets of $\Omega$.  This follows from the Harnack inequality and is used in the proof of Theorem \ref{SmoothCompactThm} -- see (B.3) and (B.5) of~\cite{cm24} and also \cite[Theorem 1.1]{sol1}.
\end{rem}

Suppose that $\set{\Omega_i}$ exhausts $\Omega$ and that $\Sigma_i$ are
properly embedded smooth minimal surfaces in $\Omega_i$.  In light of 
Theorem~\ref{SmoothCompactThm},
we define the \emph{regular} points of the sequence $\mathcal{S}=\set{\Sigma_i}$ to be the set of points
\begin{equation*}
 \reg(\mathcal{S}):= \set{p\in \Omega: \exists \rho>0 \mbox{ s.t. } \limsup_{i\to \infty} \sup_{B_\rho(p)\cap \Sigma_i} |A_{\Sigma_i}| <\infty }
\end{equation*}
and the \emph{singular} points of $\mathcal{S}$ to be the set
\begin{equation*}
\sing (\mathcal{S}):= \set{p\in \Omega: \forall \rho>0,  \liminf_{i\to \infty} 
\sup_{B_\rho(p)\cap \Sigma_i} |A_{\Sigma_i}| =\infty }.
\end{equation*}
Clearly, $\reg(\mathcal{S})$ is an open subset of $\Omega$ while 
$\sing(\mathcal{S})$ is closed in $\Omega$.  In general, $\sing 
(\mathcal{S})$ is a strict subset of $\Omega\backslash \reg(\mathcal{S})$. 
However, an elementary argument -- see \cite[Lemma I.1.4]{cm24} -- implies that 
there is a subsequence $\mathcal{S}'$ of $\mathcal{S}$ so that 
$\Omega=\reg(\mathcal{S}')\cup \sing (\mathcal{S}')$.  From now on we consider 
only sequences for which this decomposition holds.

We say that $\LL$ is the \emph{limit lamination} of $\mathcal{S}$ if for some 
$\alpha>0$, $\Sigma_i \to \LL$ in $C^{\infty,\alpha}_{loc}(\reg(\mathcal{S}))$. 
Theorem \ref{SmoothCompactThm} implies that, up to passing to a subsequence, any 
sequence $\mathcal{S}$, possesses a limit lamination $\LL$. 
Inspired by \cite{wh10}, we make the following definition:
\begin{defn}\label{minimalsurfseqdef}
We say a quadruple $(\Omega, K, \mathcal{L}, \mathcal{S})$ consisting of
\begin{enumerate}
 \item A  Riemannian three-manifold $\Omega$ exhausted by $\set{\Omega_i}$;
 \item A closed set $K\subset \Omega$;
 \item A proper smooth minimal lamination $\mathcal{L}$ in $\Omega\backslash 
K$; and
 \item A sequence $\mathcal{S}=\set{\Sigma_i}$ of properly embedded minimal 
surfaces $\Sigma_i$ in $\Omega_i$,
\end{enumerate}
is a \emph{minimal surface sequence} if 
\begin{enumerate}
\item  $\sing(\mathcal{S})=K$; and
\item  $\Sigma_i\backslash K$ converge in $C^{\infty,\alpha}_{loc}(\Omega\backslash K)$ to $\mathcal{L}$ for some $0<\alpha<1$.
\end{enumerate}
If all the surfaces in $\mathcal{S}$ are disks, then we say this is a 
\emph{minimal disk sequence}. 
\end{defn}

The work of Colding and Minicozzi \cite{cm21, cm22, cm23, cm24} implies that if 
$(\Omega, K, \mathcal{L}, \mathcal{S})$ is a minimal disk sequence, then $K$ and 
$\mathcal{L}$ have a great deal of structure. We say a leaf $L$ of $\LL$ 
is \emph{regular at $p\in K$} if $p\in \overline L$, the closure 
of $L$ in $\Omega$, and there is an $r>0$ so that $\mathcal{B}_r(p)\cap 
\overline{L}$ is an embedded smooth surface proper in $\mathcal{B}_r(p)$. Then 
the 
following holds.
\begin{prop}\label{CM}
If $(\Omega, K, \mathcal{L},\mathcal{S})$ is a minimal disk sequence, then there is an embedded one-dimensional Lipschitz curve $K'\subset\Omega$ such that $K\subset K'$ and
\begin{enumerate}
\item If $p\in K$, then there is a leaf $L\in \LL$ which is regular at $p$;
\item If $L$ is regular at $p\in K$, then 
$\overline L$ meets $K$ transversely -- in the strong sense that $\overline L$ 
meets $K'$ transversely at $p$.  
\end{enumerate}
\end{prop}
 White \cite{wh10} has shown that the regularity of $K'$ can be taken to be 
$C^1$.  Furthermore,  Meeks \cite{me25} has shown that if $K=K'$, then $K'$ can 
be taken to be $C^{1,1}$.

We will consider also the following related objects. 
\begin{defn}\label{MinConfigDef}
 We say a quadruple $(\Omega, K,\mathcal{L}, \mathcal{L}_0)$ consisting of
 \begin{enumerate}
  \item A connected open subset $\Omega\subset N$;
  \item A  closed set $K\subset \Omega$;
  \item A smooth minimal lamination $\LL_0$ in $\Omega\backslash K$; and
  \item A smooth proper minimal lamination $\LL$ in $\Omega\backslash K$,
 \end{enumerate}
 is a \emph{minimal surface closure } provided,
 \begin{enumerate}
  \item For all $p\in K$ and $\rho>0$, $\sup_{L\in \set{\LL_0}} 
\sup_{B_\rho(p)\cap L} |A_L|=\infty$; and
  \item $\overline{\LL}_0=\LL$, where here $\overline{\LL}_0$ is the closure of $\LL_0$ in $\Omega\backslash K$.
 \end{enumerate}
If all the leaves of $\LL_0$ are disks, then we say this is a \emph{minimal disk closure}. 
 \end{defn}
The limit laminations $\LL$ of minimal disk sequences and of minimal disk closures share many properties. Therefore, it is convenient to introduce the following definition.
\begin{defn} A smooth minimal lamination $\mathcal{L}$ in a Riemannian 
three-manifold $\Omega$ is a \emph{simple minimal lamination in 
$\Omega$} if there is a relatively closed set $K\subset \Omega$ and either
\begin{enumerate}
 \item $(\Omega, K, \mathcal{L}, \mathcal{S})$ is a minimal disk sequence for some $\mathcal{S}$; or
 \item $(\Omega, K, \mathcal{L},\mathcal{L}_0)$ is a minimal disk closure  for some $\mathcal{L}_0$.
\end{enumerate}
\end{defn}


\section{Simple Lifts}
In order to proceed we will need a technical definition. 
\begin{defn}
 Let $\Sigma$ be an embedded surface in a fixed Riemannian three-manifold 
$\Omega$. The surface $\Sigma$ has the \emph{simple lift property} 
if, for any $\delta>0$, $\gamma:[0,1]\to \Sigma$ a piece-wise $C^1$ curve, and 
open pre-compact subset $U\subset \Sigma$ with $\gamma\subset U$, there exist:
 \begin{enumerate}
  \item A constant $\epsilon=\epsilon(U,\delta)>0$; 
  \item An embedded minimal disk $\Delta$ in $\Omega$; and
   \item A $\delta$-lift of $\gamma$, $\widehat{\gamma}:[0,1]\to 
\mathcal{N}_\delta(U)$, 
  \end{enumerate}
such that 
  \begin{enumerate}
    \item $\widehat{\gamma}\subset \Delta\cap \mathcal{N}_{\epsilon }(U)$;
  \item $\Delta\cap \mathcal{N}_{\epsilon} (U)$ is a $\delta$-graph over $U$; 
\item The connected component of $\Delta\cap \mathcal{N}_{\epsilon} (U)$ 
containing $\widehat \gamma$ is a $\delta$-cover of $U$.
 \end{enumerate}
 Such $\widehat \gamma$ is called a \emph{simple $\delta$-lift} of $\gamma$ into 
$\Omega$. 
\end{defn}
If $\Sigma$ has the simple lift property in $\Omega$ and $\gamma$ is a curve in 
$\Sigma$, then $\gamma$ has the \emph{embedded lift property} if there is a 
a $\delta_0>0$ so that for all $\delta_0>\delta>0$, all simple $\delta$-lifts 
of $\gamma$ are embedded. Clearly, if $\gamma$ is an embedded curve, then it 
has the embedded lift property.

Throughout this paper we study the topology of minimal surfaces with the simple lift property. This is relevant to the study of the topology of a leaf of a simple minimal lamination thanks to the following proposition.
\begin{prop}\label{simpleliftprop}
Leaves of a simple minimal lamination in $\Omega$ have the simple lift property.
\end{prop}
\begin{proof}
We first consider the case of a minimal disk closure. If $L$ is a leaf of $\LL$ which is a disk, then any curve in $L$ is its own simple $\delta$-lift in any pre-compact open set of $L$ containing such curve and there is nothing to prove. If $L$ is not a disk, then $L\subset \mathcal{L}\backslash \mathcal{L}_0$.  Hence, for any point $p\in L$ there exists a sequence of points $p_i\in \LL_0$ so that $p_i\to p$. 

Note first that the definition of smooth lamination -- specifically the 
existence of Lipschitz straightening maps -- implies that for each pre-compact 
open subset $U$ of $L$ there is a constant $C=C(U)$ so that if $1>C\lambda>0$,  
then, for each leaf $L'$ of $\LL_0$, $\mathcal{N}_\lambda(U)\cap 
L'$ is a -- possibly empty -- $C\lambda$-graph over $U$. Given a curve 
$\gamma\colon [0,1]\to L$  and $U$ some pre-compact open subset of $L$ so that 
$\gamma\subset U$, let $l$ denote the length of $\gamma$ and let $d$ 
denote the diameter of $U$.  For any $\delta>0$, choose $\epsilon>0$ such that 
$C\epsilon<\min \{1,\delta\}$. Let $\mu=\frac 34 e^{-2C(l+d)}$ and pick 
$L_\mu$ to be a leaf of $\mathcal L_0$  which satisfies
$\mathcal{N}_{\mu\epsilon}(p)\cap L_\mu\neq \emptyset$ where here 
$p=\gamma(0)$. Let $\Gamma$ be a component of $L_\mu\cap \mathcal 
{N}_{\epsilon}(U)$ which contains a point $q\in 
\mathcal{N}_{\mu\epsilon}(p)\cap\Gamma$. 

The leaf $L_\mu$ is, by definition, a disk and we have chosen $\epsilon>0$ so 
that $L_\mu\cap \mathcal{N}_{\epsilon}(U)$ is a $\delta$-graph over $U$. We 
claim that $\Gamma$ is a $\delta$-cover of $U$ containing a $\delta$-lift of 
$\gamma$. This follows by showing that any curve in $U$ of length at most 
$2(l+d)$ starting at $p$ has a lift in $\Gamma$ starting at 
$q$. By construction, this lift is necessarily a $\delta$-lift. Indeed, 
if $\sigma:[0,T]\to U$ is parameterized by arclength and 
$\widehat{\sigma}:[0,T']\to \Gamma$ satisfies 
$\Pi_\Sigma(\widehat{\sigma}(t))=\sigma(t)$ for some $0<T'\leq T$ , then
\begin{equation*}
 \left|\frac{d}{dt} \dist^\Omega (\sigma(t),\widehat{\sigma}(t))\right|\leq 
C\dist^\Omega(\sigma(t),\widehat{\sigma}(t))
\end{equation*}
and so
\[
 \dist^\Omega (\sigma(t),\widehat{\sigma}(t))\leq 
e^{C t} \dist^\Omega (p,q)<\epsilon \mu e^{C t} <\epsilon.
\]
Where we used that $t\leq T\leq l+d$ to obtain the final inequality.
 Furthermore, if $t<T$, then the lift $\widehat{\sigma}(t)$ may be extended past 
$t$ provided  $\dist^\Omega (\sigma(t), \widehat{\sigma}(t))<\epsilon$. This 
proves that the leaf of a minimal disk closure has the simple lift property as 
claimed.

In the case of a minimal disk sequence, the argument is identical to the one 
above except that it uses the Harnack inequality to obtain the bound on the 
Lipschitz norms of straightening maps.  We refer to Remark 
\ref{straighteningrem} and to ~\cite[Appendix B]{cm24} for the details on how to 
obtain this bound. 
\end{proof}

A surface with the simple lift property is one for which, in an effective sense, the universal cover of the surface can be properly embedded as a minimal disk near the surface. For this reason, to understand the topology of the surface, it is important to understand the lifting behavior of closed curves.  With this in mind, we give the following definition.
\begin{defn}
 Let $\Sigma\subset \Omega$ be an embedded minimal surface with the simple lift 
property. If $\gamma\colon [0,1]\to \Sigma$ is a piece-wise $C^1$ closed curve, 
then $\gamma$ has the \emph{open lift property} if there  exists a $\delta_0>0$ 
so that, for all  $\delta_0>\delta>0$, $\gamma$ does not have a closed simple 
$\delta$-lift $\widehat{\gamma}:[0,1]\to \mathcal{N}_\delta(\Sigma)$.  Otherwise, 
$\gamma$ has the \emph{closed lift property}.
\end{defn}
If a closed curve $\gamma$ has the closed lift property, then there is a sequence $\delta_i\to 0$ so that there are closed simple $\delta_i$-lifts $\widehat{\gamma}_i$ of $\gamma$.  If it is possible to choose these lifts to be embedded we say $\gamma$ has the \emph{embedded closed lift property}.

The next lemma says that if two loops satisfying certain geometric conditions 
have the open lift property, then their commutator has the closed lift property. 
Very roughly speaking, it does this by constructing an ``effective'' 
homomorphism from the space of loops in the leaf to $\mathbb Z$.  Indeed, a 
curve $\widehat \gamma$ is a lift of $\gamma$ if and only if $\widehat \gamma 
^{-1}$ 
is a lift of $\gamma^{-1}$. Thus, $\gamma$ has the closed  lift property if and 
only if $\gamma^{-1}$ does.
\begin{prop}\label{CommutatorClosedLiftProp}
  Let $L\subset \Omega$ be an embedded minimal surface with the simple lift 
property and let
  \[   
     \alpha: [0,1]\to  L \mbox{ and } \beta:[0,1]\to  L          
        \]
 be closed piece-wise $C^1$ curves satisfying the following properties:
 \begin{enumerate}
\item Both $\alpha$ and $\beta$ have the open lift
property; 
  \item $\alpha\cap \beta=\set{p_0}$ where $p_0=\alpha(0)=\beta(0)$;
 \item There exists a two sided pre-compact open set $U\subset L$ with 
$\alpha\cup\beta\subset U$.
 \end{enumerate}
 Then the curve  $\mu=\alpha\circ\beta \circ\alpha^{-1}\circ \beta^{-1}$ has the 
closed lift property.
If, in addition, both $\alpha$ and $\beta$ have the embedded lift property,
then either $\mu$ has the embedded closed lift property, or one of 
the following two curves has the embedded closed lift property:
\[
\alpha\circ\beta  \quad \text{or} \quad \beta \circ\alpha^{-1}.
\]  
\end{prop}
\begin{proof}
  Consider a sequence $\delta_i\to 0$. 
 As $\Sigma$ has the simple lift property, there exist constants $\epsilon_i$, 
embedded minimal disks $\Delta_i$ and simple $\delta_i$-lifts of $\mu$, 
$\widehat{\mu}_i$, so that $\Delta_i\cap \mathcal{N}_{\epsilon_i}(U)$ is a 
$\delta_i$-graph over $U$, $\widehat{\mu}_i\subset \Delta_i$ and the connected 
component $\Gamma_i$ of  $\Delta_i\cap \mathcal{N}_{\epsilon_i}(U)$ containing 
$\mu_i$ is a $\delta_i$-cover of $U$.  By reparameterizing appropriate 
restrictions of $\widehat{\mu}_i$, we obtain lifts of $\alpha$, $\alpha^{-1}$, 
$\beta$  and $\beta^{-1}$.  We write 
$\widehat{\mu}_i=\widehat{\alpha}_i\circ\widehat{\beta}_i \circ\widehat{\alpha}^{-1}_i\circ 
\widehat{\beta}^{-1}_i$ where the $\widehat{\alpha}_i, 
\widehat{\beta}_i,\widehat{\alpha}^{-1}_i, \widehat{\beta}^{-1}_i\colon [0,1]\to \Gamma_i$ 
are lifts of the $\alpha,\beta, \alpha^{- 1},\beta^{-1}$.

Setting $p=\mu(0)$ we pick a small simply-connected neighborhood $V$ of $p$ that satisfies $V\subset U$. Because the $\Delta_i$ are embedded there is a natural way to order by height the components of $\Pi_L^{-1}(V)\cap \Delta_i$.  We denote these ordered components by $\widehat{V}_i(1), \ldots, \widehat{V}_i(n_i)$.  Let $\widehat{p}_i(0)=\widehat{\alpha}_i(0)$, $\widehat{p}_i(1)=\widehat{\alpha}_i(1)=\widehat\beta_i (0)$, $\widehat{p}_i(2)=\widehat{\beta}_i(1)=\widehat \alpha^{-1}_i (0)$, $\widehat{p}_i(3)=\widehat{\alpha}^{-1}_i(1)=\widehat \beta^{-1}_i (0)$, $\widehat{p}_i(4)=\widehat{\beta}^{-1}_i(1)$ and notice that $\widehat{p}_i(j)\in \widehat{V}_i(l(i,j))$ for some function $l$.  Let 
\begin{align*}
m_i[\alpha]&=l(i,1)-l(i,0),\\
m_i[\beta] &=l(i,2)-l(i,1), \\
m_i[\alpha^{-1}]&=l(i,3)-l(i,2),\mbox{ and }\\
m_i[\beta^{-1}] &=l(i,4)-l(i,3), \\
\end{align*}
 represent the (signed) number of sheets between the end points of 
$\widehat{\alpha}_i$, $\widehat{\beta}_i$, $\widehat{\alpha}^{-1}_i$ and 
$\widehat{\beta}^{-1}_i$.  As both $\widehat{\alpha}_i$ and $\widehat{\beta}_i$ are open 
lifts these numbers are never zero. 
We now prove that $m_i[\alpha]=-m_i[\alpha^{-1}]$ and 
$m_i[\beta]=-m_i[{\beta^{-1}}]$ and, hence, $\hat{\mu}_i$ is closed.  
We consider two 
cases: $m_i[\alpha]  m_i[\beta]>0$ and  
$m_i[\alpha] m_i[\beta]<0$. 

In the first case we assume, with out loss of generality, that  
$m_i[\alpha], m_i[\beta]>0$.  Using the fact that the $\Delta_i$ are embedded 
and that $U$ is two-sided, we see that there is a disjoint family of 
``parallel'' lifts of $\alpha$ which we denote by $\widehat{\alpha}_i[j]$.
The first member of the family is 
$\widehat{\alpha}_i[0]=\widehat{\alpha}_i$ and the subsequent members of the 
family are 
the lifts, $\widehat{\alpha}_i[k]$ of $\alpha$ which satisfy 
$\widehat{\alpha}_i[k](0)=l(i,0)+k$. Namely, the lift $\widehat{\alpha}_i[k]$ 
starts 
$k$ sheets above  $\widehat{\alpha}_i(0)$. By the embeddedness of $\Delta_i$ 
and 
the two-sidedness of $U$, the signed number of graphs 
between $ \widehat{\alpha}_i[0](t)$ and $\widehat{\alpha}_i[k](t)$ is 
constant in $t$.  Hence, $\widehat{\alpha}_i[
k](1)=l(i,1)+k$ -- that is, the lifts have endpoint $k$ sheets above the 
endpoint of $\widehat{\alpha}_i$.  Clearly, 
the $\widehat{\alpha}_i[k]$ are well-defined as long as $k\leq  
m_i[\beta]$.  Furthermore,  $\widehat{\alpha}_i[m_i[\beta]]$  has end point 
which is the same as the end point of $\widehat{\beta}_i$.  Hence,  since 
$\widehat{\alpha}_i[m_i[\beta]]^{-1}$ is a lift of $\alpha^{-1}$ starting at $\widehat 
\beta_i(1)$, the lift  $\widehat{\alpha}_i[m_i[\beta]]^{-1}$ must be 
$\widehat{\alpha}_i^{-1}$. That is, $m_i[\alpha]=-m_i[\alpha^{-1}]$.  An 
identical argument shows that $m_i[\beta]=-m_i[\beta^{-1}]$.

In the second case, we may suppose with out loss of generality that 
$m_i[\alpha]>0$ and $ m_i[\beta]<0$.  We first assume that 
$m_i[\alpha]+m_i[\beta]+m_i[\alpha^{-1}]\geq 0$ and obtain a contradiction.  
Under this hypothesis the end point of $\widehat{\alpha}_i^{-1}$ is not below that 
of the initial point of $\widehat{\alpha}_i$. As in the preceding argument, we can 
construct a family of ``parallel'' lifts of $\alpha^{-1}$. The first member of 
this family is $\widehat{\alpha}_i^{-1}$, the second lift in the family starts just 
below $\widehat{\alpha}_i^{-1}$ and the last lift in such family,  $\widehat{\alpha}'_i$ 
has end point that is the initial point of $\widehat{\alpha}_i$.  As before, the 
number of graphs between the start and end points of a lift 
in this family is constant equal to $m_i[\alpha^{-1}]$. Since $\widehat{\alpha}'_i$ 
has end point that is the initial point of $\widehat{\alpha}_i$, the lift 
$(\widehat{\alpha}'_i)^{-1}$ must be $\widehat{\alpha}_i$. This implies that  
$m_i[\alpha] =-m_i[\alpha^{-1}]$. Since  $m_i[\beta]< 0$, this leads to a 
contradiction.  Hence,  $m_i[\alpha]+m_i[\beta]+m_i[
\alpha^{-1}]< 0$.  Again one constructs a family of parallel lifts starting 
from $\widehat{\alpha}_i$ and ending with a curve with initial point the end point 
of $\widehat{\alpha}_i^{-1}$ which again implies that  
$m_i[\alpha]=-m_i[\alpha^{-1}]$. The same argument shows that in this case
$m_i[\beta]=-m_i[\beta^{-1}]$.

Finally, we note that if $\alpha$ and $\beta$ have the embedded lift property, 
then, because they meet at only one point, the curves $\widehat{\alpha}_i\circ 
\widehat{\beta}_i$,  $\widehat{\beta}_i \circ \widehat{\alpha}_i^{-1}$ and 
$\widehat{\alpha}_i^{-1}\circ \widehat{\beta}_i^{-1}$ are all embedded. Hence, the only 
way that $\widehat{\mu}_i$ can fail to be embedded is if one of the the first two is 
closed. 

\end{proof}

\section{Main Proof}
Rather than prove Theorems \ref{BeautifulThm} and \ref{BeautifulThm2} directly 
we prove slightly more general results.  To do so we will  restrict the 
geometry of the three-manifolds $\Omega$ we consider.
\begin{assumption} \label{MainAssump}
Let $\Omega$ be the interior of a complete, oriented, three-manifold with (possibly empty)
boundary $N=\overline{\Omega}$ satisfying the following properties:
\begin{enumerate}
 \item \label{MainAssump1} The boundary of $N$ is mean-convex;
 \item \label{MainAssump2} There is an exhaustion $\set{\Omega_t}_{t\in 
[0,1)}$ of $\Omega$ so that each $\Omega_t$ is pre-compact in $N$, $\partial 
\Omega_t$ 
is mean-convex and $\set{\partial 
\Omega_t}_{t\in(0,1)}$ foliates $\Omega\backslash \overline{\Omega}_0$;
 \item \label{MainAssump3} $\Omega$ contains no closed minimal surfaces.
\end{enumerate}
\end{assumption}
Note that $N$ does not have to be compact. For instance, $\mathbb H^3$ satisfies Assumption~\ref{MainAssump}. 
Minimal surfaces in such an $\Omega$ automatically satisfy a certain type of 
uniform isoperimetric inequality.  This follows immediately from work of B. 
White \cite{WhiteIsoPerim}. 
\begin{lem}\label{IsoperimLem}
 If $\Omega$ satisfies the conditions of 
Assumption~\ref{MainAssump}, $U\subset \Omega$ is a precompact subset of 
$\Omega$ and $\Sigma$ is a compact minimal surface with boundary, 
$\partial\Sigma\subset U$, then 
\begin{enumerate}
 \item There is a pre-compact open set $U'$ depending only on $U$ so that 
$\Sigma 
\subset U'$;
\item There is an increasing function, $\Psi_U:\Real^{\geq 0}\to \Real^{\geq 
0}$, depending only on $U$ and satisfying $\Psi_U(0)=0$ and 
\[
 |\Sigma|\leq \Psi_U(|\partial \Sigma|).
\]
\end{enumerate}
\end{lem}
\begin{proof}
By Assumption \ref{MainAssump}, there is a pre-compact open subset 
$U'=\Omega_{t_0}$ with $U\subset \Omega_{t_0}$.  As $\Omega\backslash 
\Omega_{t_0}$ is foliated by mean-convex subsets, the strong maximum 
principle and the fact that $\partial \Sigma \subset \Omega_{t_0}$ implies that 
$\Sigma\subset \Omega_{t_0}$.  Finally, the existence of the function $\Psi_U$ 
follows immediately from  \cite[Theorem 2.1]{WhiteIsoPerim} applied to 
$\overline{U'}$.
\end{proof}

Under Assumption \ref{MainAssump} we have the following gluing property that 
allows us 
to ``fill in'' curves with the embedded closed-lift property.

\begin{lem} \label{Gluelem}
Suppose that $\Omega$ satisfies Assumption~\ref{MainAssump}.
If $L\subset \Omega$ is an embedded minimal surface with the simple lift 
property and $\gamma:[0,1]\to L$ has the embedded closed lift property, then 
there exists a smooth minimal surface $\Delta$ properly embedded in $\Omega\backslash \gamma$ so that: 
\begin{enumerate}
 \item \label{GL1} $\Delta$ has finite area and is contained in a compact 
subset of 
$\Omega$;
 \item \label{GL3}$\gamma =\overline{\Delta}\backslash 
\Delta$ and $\Delta\cup\gamma$ is connected;
 \item \label{GL4}$\Delta \cap L$ is a non-empty open and 
closed subset of $L\backslash \gamma$; and
 \item \label{GL5}If $\gamma$ is embedded, then $\Delta$ is a disk.
\end{enumerate}
\end{lem}
\begin{proof}
 By hypothesis, there exists a sequence of closed embedded simple $\frac 
1n$-lifts  $\gamma_n$ of $\gamma$.  Hence, each $\gamma_n$ bounds a minimal disk 
$\Delta_n'$ inside of the minimal disks $\Delta_n$.  Clearly, there is a 
fixed pre-compact subset $U$ of $\Omega$ containing $\gamma$ together with all 
of the $\gamma_n$.  Furthermore, the length of each $\gamma_n$ is 
bounded by twice the length of $\gamma$ and so there is a uniform bound on 
$|\partial \Delta_n'|$. As a consequence, Lemma \ref{IsoperimLem} implies 
that there is a precompact subset $U'$ of $\Omega$ so that the sequence of 
disks $\Delta_n'$ are contained in $U'$ and, moreover, have uniformly bounded 
area.  A result of Schoen and Simon  \cite{Schoen1983a} -- see Theorem  
\ref{SScurvEstThm} for the statement -- then gives uniform curvature bounds for 
the 
$\Delta_n'$ on compact subsets of $\Omega\backslash 
\gamma_n$.  Hence, up to passing to a subsequence, 
Theorem~\ref{SmoothCompactThm} and the area bounds imply that the 
$\Delta_n'$ converge in $C^\infty_{loc}({\Omega}\backslash \gamma)$ to a 
properly embedded minimal surface $\Delta \subset \Omega\backslash \gamma$ of 
finite area and bounded curvature on compact 
sets of ${\Omega}\backslash \gamma$. As each $\Delta_n'$ is contained in $U'$, 
it follows that $\Delta\subset \overline{U}'$.  This proves Items \eqref{GL1}.  
Item \eqref{GL3} follows from the set theoretic convergence of 
$\bar{\Delta}_n'$.

As the $\Delta_n$ are contained in $\frac 1n$ graphs over a fixed neighborhood 
$V$ of $\gamma$,  $\Delta\backslash \gamma$ contains a component of 
$V\backslash \gamma$ and so $\Delta \backslash 
\gamma $ is non-empty and is contained in $L$. Indeed, the nature of the 
convergence implies that $\Delta \backslash 
\gamma $ is an open and closed  subset of $L\backslash \gamma$. 

If $\gamma$ is embedded, then the $\widehat{\gamma}_n$ converge to $\gamma$ 
with multiplicity one.  Clearly, to prove Item \eqref{GL5} it  
suffices to show that the $\Delta_n'$ also converge to $\Delta$ with 
multiplicity one. This is most conveniently done using the language of 
varifolds -- we refer to  \cite{si1} for details.

Let $V_{n}$ be the integer multiplicity rectifiable varifold associated to 
$\overline{\Delta}_n'$.  As $\gamma$ is piecewise smooth, the nature of the 
convergence of $\gamma_n$ toward $\gamma$ implies that $\overline{\Delta}_n'$ has 
uniformly bounded first variation. Indeed, since $\Delta_n$ is minimal, the 
first variation measure of $V_n$ is supported along $\gamma_n$.  As the $V_n$ 
also have uniformly bounded mass, we may apply the compactness theory for 
integer multiplicity rectifiable varifolds to see that (up to passing to a 
further subsequence) $V_n$ converges in the sense of varifolds (in $\Omega)$ to 
an integer multiplicity rectifiable varifold $V$.  It follows from the nature 
of the convergence that $\spt V=\overline{\Delta}$ and, moreover, at any point of 
$\Delta$ the multiplicity of $V$ is a positive integer.  Moreover, the first 
variation measure of $V$ is supported along $\gamma$.  Notice that as $L$ is a 
smooth minimal surface and $\spt V\subset L$, the constancy theorem implies 
that the multiplicity of each component of $\spt V\backslash \gamma$ is 
constant. 

Fix a point $p\in \gamma$ and a small open neighborhood $W\subset L$ about 
$p$ choosen small enough so that $\gamma$ divides $W$ into exactly two 
components $W_-$ and $W_+$.  If both $W_-$ and $W_+$ meet $\spt V$, then it 
follows from the strong unique continuation princple for smooth minimal 
surfaces that $\spt V$ is a closed minimal surface in $\Omega$.  This violates 
Assumption \ref{MainAssump} and so -- up to relabelling -- we may assume that 
$W_-\cap \spt V=\emptyset$.  As $\gamma_n$ converge to $\gamma$ with 
multiplicity one, the nature of the convergence of $\Delta_n'$ to $\Delta$ then 
immediately implies that the multiplicity of $V$ is one which proves 
the claim.
\end{proof}
\begin{cor}\label{separating}
Suppose $\Omega$ satisfies  Assumption~\ref{MainAssump}.
If $L\subset \Omega$ is an embedded minimal surface with the simple lift 
property and $\gamma$ is a closed embedded curve in $L$ with the 
closed lift property, then $\gamma$ is separating. 
\end{cor}
\begin{proof}
Let $\Delta$ be the surface given by Lemma \ref{Gluelem}.  If $\gamma$ is 
non-separating, then $L\backslash \gamma$ is connected and so, by the previous 
lemma, $L\backslash \gamma\subset \Delta$.  Therefore, by strong unique 
continuation for smooth minimal surfaces, $\overline{\Delta}=\Delta\cup 
\gamma=\Delta\cup L$ is  a closed minimal surfaces of finite area in  
$\Omega$   which contradicts Assumption \ref{MainAssump}.
\end{proof}

We are now in a position to prove that surfaces with the simple lift property in 
regions satisfying Assumption \ref{MainAssump} must have genus zero. 
\begin{prop} \label{GenusZeroProp}
Suppose $\Omega$ satisfies  Assumption~\ref{MainAssump}. If $L\subset \Omega$ is an embedded minimal surface with the simple lift property in $\Omega$, then $L$ has genus zero.
\end{prop}
\begin{proof}
 Arguing by contradiction, suppose that $L$ has genus greater than zero. Then, by the classification of surfaces, there exist two piece-wise smooth, non-separating Jordan curves $\alpha,\beta: [0,1]\to L$ and a two-sided pre-compact set $U\subset L$ such that the following holds:
 \begin{itemize}
\item $\alpha\cap \beta= p=\alpha(0)=\beta(0)=\alpha(1)=\beta(1)$;
\item $L\backslash (\alpha\cup \beta)$ is connected;
\item $\alpha\cup\beta\subset U$.
\end{itemize}   

By Lemma \ref{separating} both curves have the open lift property and so   
$\alpha$ and $\beta$ satisfy  the hypothesis of Proposition 
\ref{CommutatorClosedLiftProp}.   Hence, there exists a sequence of closed 
curves $\nu_n$ that are closed simple $\frac 1n$-lifts of  $\nu 
=\alpha\circ\beta\circ\alpha^{-1}\circ\beta^{-1}$.  Proposition 
\ref{CommutatorClosedLiftProp} further tells us that either the curve 
$\mu=\nu$,  
the curve $\mu=\alpha\circ\beta$  or  the the curve $\mu=\beta\circ \alpha^{-1}$ 
has the embedded closed lift property.  In all cases, Lemma \ref{Gluelem} 
gives a  minimal surface $\Delta$ properly embedded  in ${\Omega}\backslash 
\mu$ and that $[\Delta\backslash \mu]\cap L$ is a non-empty, open and closed 
subset of $L\backslash \mu$. As $L\backslash \mu$ is connected, 
$\Delta=L\backslash \mu$.   Hence, by the strong unique continuation property 
of 
smooth minimal surfaces, $\overline{\Delta}=\Delta\cup L$ is a properly 
embedded minimal surface 
in $ \Omega$ of finite area and we contradict Assumption \ref{MainAssump}.
 \end{proof}

Clearly, Theorem \ref{BeautifulThm2} follows from
Proposition \ref{simpleliftprop}.  Indeed, the region 
$\Omega$ of  Theorem  \ref{BeautifulThm2} can be seen to satisfy Assumption 
\ref{MainAssump} by taking the exhaustion to be $\Omega$ itself.

\section{Minimal Disk Sequences}

In this section we conclude the proof of Theorem~\ref{BeautifulThm}. 
We first show that the leaves of the limit lamination do not have many holes.
 \begin{prop}\label{NoPantsProp}
 	If $(\Omega, K, \mathcal{L}, \mathcal{S})$ is a minimal disk sequence  and $L$ 
 	is a  leaf of $\LL$, then
 	\begin{enumerate}
 		\item  If $L$ is two-sided, then $L$ is either a disk or an annulus;
 		\item If $L$ is one-sided, then $L$ is a M\"obius band.
 	\end{enumerate}
 \end{prop}
 \begin{proof}
 	We will argue by contradiction. 
 	For any three separating Jordan curves $\gamma_1$,  $\gamma_2$ and $\gamma_3$ in 
 	$L$  with the property that no one of the curves separates the other two  it 
 	is the case that $L\backslash \left(\gamma_1\cup \gamma_2\cup 
 	\gamma_3 \right)$ has four components $L_1,L_2,L_3, L_4$. Label the $L_i$ so  
 	that $\overline{L}_1\cap (\gamma_2\cup \gamma_3)=\emptyset$, 
 	$\overline{L}_2\cap(\gamma_1\cup \gamma_3)=\emptyset$, $\overline{L}_3\cap (\gamma_1\cup 
 	\gamma_2)=\emptyset$ and   $\gamma_1\cup \gamma_2 \cup \gamma_3\subset 
 	\overline{L}_4$.    By Proposition~\ref{GenusZeroProp}, 
 	$L$ has genus zero. Hence, if 
 	$L$ is  neither a disk, an annulus 
 	nor a M\"obius band, then the classification of surfaces implies that the 
 	$\gamma_i$ may be choosen so that 
 	\begin{itemize}
 		\item $L_1,$ $L_2$ and $L_3$ are not disks;
 		\item $L_4$ is two-sided.
 	\end{itemize}
 	[THIS PART NEEDS TO BE CHANGED]
 	We claim that for such a choice, $\gamma_1, \gamma_2$ and $\gamma_3$ have the 
 	open lift and embedded lift property. Indeed, being embedded curves, they 
 	clearly have the embedded lift property. Suppose $\gamma_i$ had the closed lift 
 	property, then  applying Lemma \ref{Gluelem} would give that $\gamma_i$ is the 
 	boundary of a disk $\Delta\subset L$ contradicting our choice of $\gamma_i$. 
 	[!!!]
 	
 	Let $\sigma$ be an embedded arc  in $L_4$ which 
 	connects $\gamma_1(0)$ to $\gamma_2(0)$.  
 	Notice that the classification of surfaces tells us that such $\sigma$ exists 
 	and does not separate $L_4$.  Consider the new closed curve 
 	$\gamma_4=\sigma^{-1}\circ \gamma_2\circ \sigma$. By an argument analogous to 
 	the one described before, this curve must also have the open lift property. In 
 	fact, the embeddedness of $\sigma$ and of $\gamma_2$ and the fact that 
 	$\gamma_2$ has the open-lift property imply that $\gamma_4$ has the 
 	embedded-lift property.

 	We now consider the closed curve $\nu=\gamma_1\circ \gamma_4 \circ 
 	\gamma_1^{-1} \circ \gamma_4^{-1}$.  Proposition \ref{CommutatorClosedLiftProp} 
 	implies that either  the curve $\mu=\nu$, the curve $\mu=\gamma_1\circ\gamma_4$ 
 	does or  the curve $\mu=\gamma_4\circ 
 	\gamma_1^{-1}$ has the embedded closed lift property.  
 	In all cases, let $\Delta$ be the embedded minimal surface given by Lemma 
 	\ref{Gluelem}. The fact that $\Delta$ is connected and that
 	$\gamma_1\cup\gamma_2\subset \overline \Delta$ together imply that $L_4\subset 
 	\Delta$. However, as $\gamma_3$ is a Jordan curve disjoint from $\nu$, $\gamma_3$ 
 	must be the limit of embedded closed curves in $\Delta_n$ -- that is, it has 
 	the embedded closed lift property. This is contradiction and proves the 
 	Proposition.
 \end{proof}

We next show that in the case of minimal disk sequences, the leaves of 
the limit lamination are two-sided. 
\begin{prop}\label{LamSeqNoTwoSidedProp}
Let $\Omega$ satisfy Assumption \ref{MainAssump}.
If $(\Omega, K, \mathcal{L}, \mathcal{S})$ is a minimal disk sequence  and $L$ 
is a  leaf of $\LL$, then $L$ is two-sided.
\end{prop}
\begin{proof} 
 Suppose that $L$ is a one-sided leaf of $\mathcal{L}$.  By Proposition
\ref{NoPantsProp}, $L$ is a M\"{o}bius band. As a consequence, there is a 
closed Jordan curve $\gamma:[0,1]\to L$ that is non-separating and so, by 
Lemma \ref{separating}, has the open lift property. Let $U$ be an 
open pre-compact neighborhood of $\gamma$ and pick $\epsilon>0$ so that 
$\overline{\mathcal{N}}_\epsilon(U)$ is a regular neighborhood. As $\gamma$ is 
non-separating, $U$ is one-sided and, indeed, the surface 
$M=\Pi_L^{-1}(\gamma)\cap \overline{\mathcal{N}}_\epsilon (U)$ is a 
closed M\"obius band.

 Let $\Sigma_i$ be the surfaces in $\mathcal{S}$. There are curves, 
$\widehat{\gamma}_i$, which are components of  $\Sigma_i\cap M$  
containing $\delta$-lifts of $\gamma$ for any $\delta$ sufficiently small. In 
particular, the $\widehat{\gamma}_i$ are proper, but not closed,  in $M$.   
Furthermore, after possibly shrinking $\epsilon$, they are monotone in the 
sense that $(\widehat{\gamma}_i')^\top \neq 0$ and $\widehat{\gamma}_i$ meets 
$\partial M$ transversely. Finally, for $i$ large enough, 
the map $\Pi_L:\widehat{\gamma}_i\to \gamma$ contains a three-fold cover. 
We 
claim this yields a contradiction. 

To see this consider $\pi: 
\widetilde{M}\to M$ the oriented double cover of ${M}$.  
As $\wt{M}$ is an annulus and $\widehat{\gamma}_i$ is monotone: 
 \begin{itemize}
 \item $\wt M= \mathbb S^1\times [-1,1]$ with coordinates $(\theta,z)$;
 \item $M= \wt M \slash_\sim$ with  $(\theta, z)\sim (\theta+\pi, -z)$;
 \item $\mathbb S^1\times \{0\}=\pi^{-1}(\gamma)$;
 \item $\wt{\gamma}_i=\pi^{-1}(\widehat{\gamma}_i)$ is a graph over $\mathbb 
S^1$.
 \end{itemize}
As $\wt{\gamma}_i$ is a graph, we may parametrize
$\wt{\gamma}_i(\theta)$ as $(\theta,v_i(\theta))$ for $\theta\in[0, T_i]$ and 
some continuous function $v_i$ with $|v_i(0)|=|v_i(T_i)|=1$ and 
$|v_i(\theta)|<1$ for $\theta\in (0,T_i)$. Since $\widehat{\gamma}_i$ contains 
a three-fold cover of $\gamma$, $T_i>3\pi$.  The embeddedness of
$\widehat{\gamma}_i$ implies that 
 for any $\theta\in [0,T_i-\pi]$, 
$v_i(\theta+\pi)\neq -v_i(\theta)$ and for any $\theta\in 
[0,T_i-2\pi]$, $v_i(\theta+2\pi)\neq v_i(\theta)$.  Without loss of generality, 
we assume that $ v_i(0)=-1$. Consider the continuous 
function $g_i$ defined for $\theta\in [0,T_i-2\pi]$ by 
$g_i(\theta)=v_i(\theta+2\pi)-v_i(\theta)$.  Notice that 
$g_i(T_i-2\pi)<0$ 
if and only if $ v_i(T_i)=-1$.  Hence, 
as  $g_i(0)>0$, the intermediate value theorem implies  
that 
$v_i(T_i)=1$. Finally, consider the continuous function 
$f_i$ defined 
for $\theta\in[0, T_i-\pi]$ by $f_i(\theta)= v_i(\theta+\pi)+v_i(\theta)$. 
Clearly,  
$ f_i(0)<0$ and $ f_i(T_i-\pi)>0$. 
Hence the intermediate value theorem contradicts the fact 
that $f_i(\theta)\neq 0$; completing the proof.
\end{proof}

We now finish the proof of Theorem~\ref{BeautifulThm}. For completeness, we recall its statement.

\begin{theoremn}[\ref{BeautifulThm}]
 Let $\Omega$ be the interior of a compact oriented three-manifold  
$N=\overline{\Omega}$ with mean-convex
boundary. If $\Omega$ contains no closed minimal surfaces and 
$(\Omega, K, \mathcal{L}, \mathcal{S})$ is a minimal disk sequence, then the 
leaves of $\LL$ are either disks or annuli.  Furthermore, if $L$ is a leaf of 
$\mathcal L$ with the property that  $\overline{L}$ -- the closure in $\Omega$  of 
$L$ -- is a properly embedded minimal surface, then $\overline{L}$ is either a disk 
or it is an annulus which is disjoint from $K$.
\end{theoremn}

\begin{proof}
We first note that $\Omega$ satisfies the conditions of Assumption 
\ref{MainAssump} by taking the exhaustion to be $\Omega$ itself. Moreover, each 
leaf of $\LL$ has the simple lift property by Proposition \ref{simpleliftprop}. 
 Hence, Propositions~\ref{NoPantsProp} and \ref{LamSeqNoTwoSidedProp} together 
imply $L$ is either a disk or an annulus. The remainder of the theorem follows 
 from the deeper result of Colding and Minicozzi that we summarized in 
Proposition~\ref{CM}. Indeed, if $L$ is a leaf of $\LL$ with $\bar{L}$ a 
properly embedded minimal surface, then it is regular at each $p\in 
\overline{L}\cap K$. Hence, by 
Proposition~\ref{CM},  $\overline L\cap 
K$ is a discrete set of points in $\overline{L}$.  As $L$ is either a disk or 
an annulus, if $\overline{L}$ is an annulus it must be disjoint from $K$.
\end{proof}

\appendix

\section{Examples}
\label{ExampleApp}
\subsection{One-sided Limit Leaf}
In this section, we construct a simply-connected minimal surface $M$ embedded in 
a solid torus that is not properly embedded. Moreover, its closure is  a 
lamination in the solid torus consisting of three leaves. The leaf $M$ and two 
limit leaves. One limit leaf is an annulus while the other is a M\"obius 
band.

Let $T$ be a solid torus obtained by revolving a disk $\mathcal D$ in the 
$(x_1,x_3)$-plane around the $x_3$-axis. We take $\mathcal D$ small enough so 
that $T$ is  
mean-convex and there exists a stable 
minimal M\"obius band $M$ embedded in $T$ whose double cover is also stable 
and with boundary a simple closed curve in 
$\partial T$.  For the existence of such a 
surface we refer to \cite{mwe1}. Since the double cover of $M$ is  stable, a 
normal neighbourhood of $M$ can be foliated by minimal surfaces and, except for 
$M$ itself, the leaves of this foliation are two-sided annuli. Let $\Sigma$ 
denote the outermost leaf of this foliation and let $W$ denote the open region  
between $M$ and $\Sigma$. Let  $\Sigma_t$, $t\in [0.1]$, be an indexing of 
the leaves of the foliation with $\Sigma_0=M$ and $\Sigma_1=\Sigma$. Let $\wt T$ 
be the universal cover of $T$ with the induced metric. We realize $\wt 
T$ as
\[
\wt T =\{(x,y,z): x^2+y^2\leq 1\}\subset \Real^3
\] in a manner so that for any $\alpha\in\R$ 
the map
\[
G_\alpha\colon \wt T\to\wt T,\quad G(x,y,z)=(x,y,z+\alpha)
\]
is an isometry and for any $t\in \R$ the set 
\[
B_t\colon =\{(x,y,z)\in \wt T \mid z= t\}
\]
 is a minimal surface that is a lift of a disk obtained by intersecting $T$ with 
a vertical plane containing the $z$-axis.  The maps $G_{2\pi n}$, $n\in\mathbb 
Z$ are the deck transforms of $\wt T$. 
Let $\Pi\colon\wt T\to T$
denote the natural projection. Given an embedded surface $S\in \wt T$, if for 
any $n\in \mathbb Z\backslash \{0\}$ it holds that $G_{2\pi n}(S)\cap 
S=\emptyset$, then 
 $\Pi(S)$ is embedded 
in $T$. 

The M\"obius band $M$ lifts to a strip $\wt M$ with boundary consisting of two 
curves in $\partial \wt T$ and, being a lift, is   invariant by the deck 
transforms $G_{2\pi n}$, $n\in\mathbb Z$. The strip  $\wt M$ is two-sided 
and separates $\wt T$ into two components. Each leaf $\Sigma_t$ of the 
foliation, $t\in (0,1]$, lifts to two strips $\wt\Sigma^+_t$ and $\wt\Sigma^-_t$ 
 on  opposite sides of $\wt M$ and this gives a foliation of a two-sided normal 
neighbourhood of $\wt M$.  We shall denote by $\wt W$ the region foliated by the 
leaves $\wt\Sigma^+_t$ and denote such leaves by $\wt\Sigma_t$. Given a point 
$p\in \wt W$, then $p\in \wt \Sigma_t$ for a certain $t\in (0,1)$ and we denote 
that $t$ by $t(p)$. Note that $\partial \wt W \cap \partial \wt T$ consists of 
two disconnected components, $\Delta_1$ and $\Delta_2$, and let 
$\alpha_i\subset \Delta_i$, $i=1,2$ be  analytic curves such that the following 
holds:
\begin{itemize}
\item  $\alpha_i$ intersects $\partial \wt \Sigma_t$, $t\in (0,1)$, in exactly one point;
\item  $\alpha_i$ converges to $\partial \wt \Sigma_1\cap\Delta_i$ as $z$ goes to infinity and  to $\partial \wt M\cap\Delta_i$ as $z$ goes to minus infinity.
\end{itemize}

Let $\wt W_n$ be the region in $\wt W$ in between the minimal disks $B_{\pm 
2\pi n}$. Then, $\wt W=\bigcup_n\wt W_n$  and $\partial\wt W_n$ 
consists of six surfaces: four minimal surfaces, $\wt \Sigma_0^n=\wt M\cap \wt 
W_n$, $\wt \Sigma_1^n=\wt \Sigma_1 \cap \wt W_n$, $B_+^n=B_{2\pi n}\cap \wt 
W_n$, and $B_-^n=B_{-2\pi n}\cap \wt W_n$ and two mean convex surfaces 
$\Delta^n_1=\Delta_1\cap\wt W_n$ and $\Delta^n_2=\Delta_2\cap\wt W_n$.  Since 
the contact angle between such surfaces is less than $\pi$, the boundary of $\wt 
W_n$ is mean-convex and a good barrier to solve Plateau problem. 

Let $\gamma_n\in\partial W_n$ be a piece-wise smooth simple closed curve 
constructed in the following way. The curve $\gamma_n$ is given by the union 
$\alpha^n_1\cup\beta^n_+\cup\alpha^n_2\cup \beta^n_-$ where 
$\alpha_i^n=\alpha_i\cap\wt W_n$. The curve $\beta^n_+$ is an arc connecting the 
endpoints, $p^n_i$ of $\alpha^n_i$ in $B_+^n$.  If $t_n=t(p^n_1)=t(p^n_2)$, 
then we take $\beta^n_+$ to lie in $\Sigma_{t_n}.$  Otherwise, we take 
$\beta^n_+$ to intersect each $\Sigma_t$ in at most one point. We choose the 
the curve $\beta^n_-$ in an analogous manner in $B_-^n$.
Clearly, by our choices of $\alpha_i$,  for $n$ sufficiently large, 
\[
\max_{p\in \beta^n_+}\{t(p)\}<\min_{p\in \beta^n_-}\{t(p)\}.
\]
This implies that for any $n$ sufficiently large and any 
$m\in\mathbb Z\backslash \{0\}$ then 
\[
 G_{2\pi m}(\gamma_n)\cap \gamma_n=\emptyset
\]

By a result in Meeks and Yau \cite{my2}, $\gamma_n$ is the boundary of an 
embedded, area minimizing disk $D_n\subset \wt W_n$. Since it is area minimizing 
and, for $n$ large, $\partial G_{2\pi m}(D_n)\cap\gamma_n=\emptyset$ for any 
$m\in\mathbb Z\backslash \{0\}$, it follows that 
\[
 G_{2\pi m}(D_n)\cap D_n=\emptyset,
\]
giving that $\Pi ({D_n})$ is also embedded. Moreover, since $D_n$ is area 
minimizing and $\alpha_i^n$, $i=1,2$, are analytic curves, it satisfies 
curvature estimates up to $\alpha_i$, $i=1,2$ and a standard compactness 
argument gives that it converges to a complete simply-connected minimal surface 
$D_\infty$ embedded in $\wt T$ with boundary $\alpha_i$, $i=1,2$. By construction, 
\[
 G_{2\pi m}(D )\cap D =\emptyset
\]
for any 
$m\in\mathbb Z\backslash \{0\}$,
therefore, if we let   
$D=\Pi(D_\infty)$, then $D$ is a complete embedded disk. Clearly it is  not 
properly embedded $T$. By curvature estimates for stable minimal 
surfaces, $\overline D$, the 
closure of $D$ in $T$, is a minimal lamination. We claim that $\overline D$ consists of three leaves, $D$ itself and two limit 
leaves $D_1$ and $D_2$. By construction, $\overline D$ contains a compact leaf 
$D_1$ with boundary $\partial \Sigma$ and a compact leaf $D_2$ with boundary 
$\partial M$. Using the foliation $\Sigma_t$ and the strong maximum principle 
one concludes that $D_1=\Sigma$ and $D_2=M$.

\subsection{Torus limit leaf}
In this section we construct a three-manifold, $\Omega$, and a complete, 
embedded disk $\Delta\subset \Omega$ whose closure, $\overline{\Delta}$, 
is a minimal lamination in $\Omega$ one of whose leaves is a minimal 
torus.   More specifically, we take $\Omega=\mathbb{T}^2\times \Real$ together 
with a certain metric for which Assumption \ref{MainAssump} does not hold.
In $\Omega$ we  construct an embedded minimal disk $\Delta$ that is 
not properly embedded and so the closure of $\Delta$ is  a 
proper minimal lamination in $\Omega$ consisting of five leaves. The leaf 
$L_1=\Delta$ and four limit leaves. Two of the limit leaves are the tori 
$L_2=\mathbb{T}^2\times \set{-1}$ and $L_3=\mathbb{T}^2\times \set{1}$, the 
other two $L_4$ and $L_5$ are non-proper annuli with 
$\overline{L}_4=\overline{L}_5=L_2\cup L_3$.  The original idea for this construction is 
due to D. Hoffman; we refer also to \cite{CalleLee} for a related construction. 

We begin by constructing a metric $g$ on the cylinders
\begin{equation*}
 C=\mathbb{S}^1_\theta \times \Real_t.
\end{equation*}
Consider the metric
\begin{equation*}
 g_0= (2+\cos \pi t) d\theta^2 + dt^2
\end{equation*}
and the foliation of $C$ by circles, $\mathbb{S}^1[t]=\mathbb{S}^1\times \set{t}\subset C$. 
It is an elementary computation to see that these circles all have constant 
curvature. Moreover, the leaves which are geodesics are $\mathbb{S}^1[i]$ for $i\in 
\mathbb Z$.  When $i\in 2\mathbb Z$ these geodesics are unstable 
while for $i\in 2\mathbb Z+1$ they are stable.  
Let $U$ be the connected 
component of $C\backslash \left( \mathbb{S}^1[-1]\cup 
\mathbb{S}^1[1]\right)$ which contains $(0,0)$. Similarly, we consider the 
foliation $\alpha_\theta=(\set{\theta}\times \Real)\cap U$ of $U$.  It is clear 
that all the leaves of this foliation are geodesics.
Finally, let us denote by $T_v$ the ``translation'' map $T_v((\theta,t))=(\theta+v,t)$ 
which is clearly an isometry and by $R$ the isometric 
involution given by $R((\theta,t))=(-\theta, -t)$.

Standard methods -- e.g., a shooting method or a minimization procedure in the 
universal cover of $C$ -- produce an embedded geodesic
$\gamma_+:[0,\infty)_s\to \gamma_+\subset U$ with $\gamma_+(0)=(0,0)$ and so 
that the $t$ coordinate of $\gamma_+(s)$ is monotonically increasing in $s$.  Here $s$ is the 
arclength parameter.  It is clear that $\gamma_+$ must accumulate at 
$\mathbb{S}^1[1]$.  Let $\gamma=\gamma_+\cup R(\gamma_+)$. This is a non-proper 
geodesic in $C$ which accumulates at $\mathbb{S}^1[-1]\cup \mathbb{S}^1[1]$.
It follows also from the construction that if $\gamma_v:=T_v(\gamma)$ then 
$\set{\gamma_v}_{v\in \Real}$ is a foliation of $U$.  With that in mind, let 
$\gamma_-=\gamma_{-\pi/2}$ and $\gamma_+=\gamma_{\pi/2}$ and let $V$ be the 
component of $U\backslash \left(\gamma_-\cup \gamma_+ \right)$ which contains $(0,0)$.

We now modify the metric $g_0$, so that geodesics which pass 
through $(0,0)$ are unstable.
To that end, pick a compactly supported function $\phi\in C^\infty_0(V)$ so 
that
\begin{itemize}
 \item $0\leq \phi\leq 1$;
 \item $\spt (\phi)\subset B_{2\epsilon} \subset V\cap 
(-\frac{\pi}{2},\frac{\pi}{2})\times (-1,1)$;
 \item $\phi\circ R=\phi$;
 \item $\phi=1$ on $B_{\epsilon}$.
\end{itemize}
Here $B_{r}$ is the geodesic ball (with respect to 
$g_0$) about $(0,0)$ of radius $r$ and we choose $\phi$ so that $2 \epsilon$ is 
smaller 
than the injectivity radius of $g_0$ at $(0,0)$.
Now fix a point $p\in \mathbb{S}^2$ and let $g_{\mathbb{S}^2}$ be the round 
metric of curvature one on $\mathbb{S}^2$. We denote by 
$\mathcal{B}_{r}$  the geodesic ball of radius 
$r$ in $\mathbb{S}^2$ about $p$.  As $B_{2\epsilon}$ and 
$\mathcal{B}_{\frac{7}{8}\pi}$ are disks, there is a smooth 
diffeomorphism 
$$\Psi:B_{2\epsilon}\to 
\mathcal{B}_{\frac{7}{8}\pi}.$$
Moreover, we may choose this 
smooth diffeomorphism so 
that
\begin{itemize}
 \item $\Psi((0,0))=p$;
 \item $\Psi(B_{\epsilon})=\mathcal{B}_{\frac{3}{4}\pi}$;
 \item $\Psi\circ R = \widetilde{R}\circ \Psi$ -- here $\widetilde{R}$ is the isometry 
of $\mathbb{S}^2$ given by rotating $180^\circ$ around the line 
through $p$ and $-p$.
\end{itemize}
We now set 
\[
 g_1= (1-\phi) g_0+ \phi \Psi^*g_{\mathbb{S}^2}
\]
Geodesics of $g_1$ that pass through $(0,0)$ are, by construction, unstable.
\begin{lem}\label{LargesLem}
If $\gamma$ is a geodesic in $V$ for $g_{1}$ such that
$(0,0)\in 
\gamma$ and $B_{\epsilon}\cap \gamma$ is proper in $B_{\epsilon}$, then $\gamma$ 
is unstable.
\end{lem}
\begin{proof}
If $\gamma$ is such a geodesic, then $\Psi(\gamma\cap B_\epsilon)$ is a  proper geodesic in $\mathcal{B}_{\frac{3}{4} 
\pi}$ that contains $p$. Hence $\Psi(\gamma\cap B_\epsilon)$,  has 
length at least $\frac{3}{2}\pi>\pi $ and so is unstable.
\end{proof}
Note that the curves $\mathbb{S}^1_{\pm 1}$, $\gamma_\pm$ and $\alpha_{\pi}$ all 
remain 
geodesics for $g_1$ as $g_1=g_0$ in a  neighborhood of these 
curves.  Furthermore, as $\phi\circ R=\phi$ and $\Psi\circ R=\tilde {R} 
\circ \Psi$,  $R$ is also an isometry of $g_1$.

Our goal now is to construct the desired disk $\Delta.$  To that end, let 
$\Omega=\mathbb{S}^1_\psi\times C=\mathbb{T}^2\times \Real$ have the product 
metric
\begin{equation*}
 g_\Omega= d\psi^2+g_{1}.
\end{equation*}
Set $U'=\mathbb{S}^1\times U$ and $V'=\mathbb{S}^1\times V$ and note that 
$\Gamma_\pm=\mathbb{S}^1\times \gamma_\pm$ are totally geodesic cylinders 
which 
accumulate at the totally geodesic tori $\mathbb{T}^2[\pm 
1]=\mathbb{S}^1\times 
\mathbb{S}^1[\pm 1]$.  Likewise,  let $A=\mathbb{S}^1\times \alpha_\pi \subset 
U'$, which is a totally geodesic annulus.
The following ``translation'' map is an isometry of 
$g_\Omega$ for $v\in \Real$
$$T_v((\psi,  \theta, t))=(\psi+v,  
\theta, t)$$
and following ``reflection'' map
\begin{equation*}
 R'((\psi,  \theta, t))=(-\psi,R(\theta, t))=(-\psi, -\theta, -t).
\end{equation*}
is an isometric involution.

We denote the universal cover of $\Omega$ by
$\widehat{\Omega}$.  That is,
\begin{equation*}
 \widehat \Omega=\Real\times \Real \times \Real
\end{equation*}
with coordinates $(\widehat {\psi},\widehat \theta, \widehat t)$.  
Let $\widehat \Pi: \widehat \Omega \to 
\Omega$ be the natural covering map.  For subsets $S\subset \Omega$ we will 
denote lifts of these sets to $\widehat \Omega$ by $\widehat{S}$.
In particular, the tori $\mathbb{T}^2[{\pm 1}]$ lift to stable minimal
disks  $\widehat{\mathbb{T}}^2[{\pm 1}]$ and the cylinders $\Gamma_\pm$ lift to 
minimal 
disks $\widehat{\Gamma}_\pm$ which together bound a region 
$\widehat{V}:=\widehat{V}'$ which contains $(0,0,0)$.  Likewise, we let 
$\widehat{A}_i=\Real\times \set{\pi+2\pi i}\times (-1,1)$ for $i\in \mathbb{Z}$ 
be lifts of $A$.
We denote by $\widehat{T}_v$ the isometry of $\widehat{\Omega}$ given 
by $\widehat{T}_v^1(\widehat \psi, \widehat \theta, \widehat t)=(\widehat 
\psi+v, \widehat \theta, \widehat t)$ for $v\in \Real$ and let $\widehat{R}$ be 
the reflection $\widehat{R}(\widehat \psi, \widehat \theta, \widehat 
t)=(-\widehat \psi, -\widehat \theta, -\widehat t)$. Note that 
$\widehat{T}_{2\pi i}$ for $i\in \mathbb Z$ are deck transforms of the cover.  
Furthermore, $\widehat{R}(\widehat{\Gamma}^+)=\widehat \Gamma^-$ and 
$\widehat{R}(\widehat{A}_i)=\widehat{A}_{-i-1}$. Finally, let us denote by 
$\widehat G_i$, $i\in \mathbb Z$, the deck transforms 
$$
\widehat G_i((\psi,  \theta, t))=(\psi,  
\theta+2\pi i, t)
$$
and note that $\widehat G_i(\widehat V)\cap \widehat V=\emptyset$ for 
$i\in\mathbb Z\backslash \{0\}$.

We now construct an 
embedded minimal disk $\widehat \Delta$ in $\widehat{\Omega}$ so that $\Delta = 
\widehat{\Pi} (\widehat{\Delta})$. To that end, let $\widehat{\sigma}_j^+$  be
the curves $\widehat{\Gamma}_+\cap \set{ \widehat{\psi}=j}$ and  
$\widehat{\sigma}^-_j$ 
the curves $\widehat{\Gamma}_{-}\cap \set{ \widehat{\psi}=-j}$. We denote by 
$\widehat{\sigma}_{j,i}^+$ the segment of $\widehat{\sigma}_j$ between $\widehat{A}_i$ and 
$\widehat{A}_{-i-1}$ and likewise for $\widehat{\sigma}_{j,i}^-$.  Let 
$\widehat{\tau}_{j,i}^-$ be 
a real-analytic curve connecting the endpoint  of $\widehat{\sigma}_{j,i}^+$ in 
$\widehat A_{-i-1}$ to the endpoint of $\widehat{\sigma}_{j,i}^-$ which is chosen to be 
contained in $\widehat A_{-i-1}$ and to have the property that both coordinates $\widehat t$ and 
$\widehat{\psi}$ are strictly monotonic. Set  $\widehat \tau_{j,i}^+=\widehat 
R(\widehat \tau_{j,i}^-)$.  
One verifies that $\widehat{\tau}_{j,i}^+$ connects the other endpoints of 
$\widehat{\sigma}_{j,i}^+$ and $\widehat{\sigma}_{j,i}^-$. Hence,  
\[
\widehat{\delta}_{j,i}:=\widehat{\sigma}_{j,i}^+\cup \widehat{\sigma}_{j,i}^- 
\cup 
\tau_{j,i}^+\cup \tau_{j,i}^-
\]
is a closed curve and 
$\widehat{\delta}_{j,i}=\widehat{R}(\widehat{\delta}_{j,i})$.  We note that our choice of 
curves implies further that $\widehat{T}_v(\widehat{\delta}_{j,i})\cap 
\widehat{\delta}_{j,i}=\emptyset$ for $v\neq 0$.  Now let $\widehat{\Delta}_{j,i}$ be 
minimal disks which solve the Plateau problem with boundary $\widehat \delta_{j,i}$. 
 By the strong maximum principle $\widehat{T}^1_v(\widehat\Delta_{j,i})\cap 
\widehat\Delta_{j,i}=\emptyset$ for $v\neq 0$. 
In particular, $\set{\widehat{T}^1_v(\widehat\Delta_{j,i})}_{v\in \Real}$ is a minimal 
foliation $\mathcal{D}_{j,i}$ of an open subset, $\widehat{V}_{j,i},$ of 
$\widehat{V}$ and $\widehat{T}_v$ leaves $\mathcal{D}_{j,i}$ 
invariant.  This together with the strong maximum principle applied to the 
Jacobi function generated by $\widehat{T}_v$ implies that the leaves of 
$\mathcal{D}_{j,i}$ are graphs over 
\[
V_i:=\{(\widehat \psi,\widehat \theta,\widehat t )\in \widehat V :\widehat \psi=0, \theta \in (-\pi-2\pi i, \pi+2\pi i)\}.
\]
As $\widehat{R}$, leaves both $\widehat \delta_{j,i}$  and 
$\widehat{V}_{j,i}$ unchanged, it follows from the strong maximum principle 
that $
\widehat\Delta_{j,i}=\widehat{R} (\widehat\Delta_{j,i})$.  In particular 
$(0,0,0)\in \widehat \Delta_{j,i}$.  By Theorem~\ref{SmoothCompactThm}, up to 
passing to a subsequence, the minimal foliations  
$\mathcal{D}_{j,i}$ converge smoothly on compact subsets of $\widehat{V}$ to a 
minimal foliation of $\widehat{V}$.  This foliation is 
also invariant under $\widehat{T}_v$.  A consequence of this is that if $L$ is a 
leaf of $\mathcal{D}$, then either $L$ splits as the product $\Real \times \eta$ 
where  $\eta$ is a geodesic in $V_{\infty}$ or $L$ is graph over some open 
subset of $V_\infty$. If the former occurs, then the stability of $L$ implies 
that $\eta$ is also a stable geodesic. 

Let $\widehat{\Delta}$ be the leaf of $\mathcal{D}$ which contains $(0,0,0)$. 
As $\widehat \Delta$ is the limit of $\widehat\Delta_{j,i}$, $\widehat 
\Delta$ is complete. By Lemma \ref{LargesLem}, the geodesic in $V_\infty$ 
through $(0,0,0)$ is unstable  and hence $\widehat \Delta$ cannot split.  In 
particular, $\widehat \Delta$ is a graph over some open subset of $V_\infty$ 
and $\widehat T_v (\widehat \Delta)\cap \widehat \Delta =\emptyset$ 
Set $\Delta=\widehat \Pi (\widehat \Delta)$ and note that, as $\widehat \Delta 
\subset \widehat V$, $\widehat G_i(\widehat \Delta)\cap \widehat 
\Delta=\emptyset$ for $i\in \mathbb Z\backslash \{0\}$. Hence, $\Delta$ is 
a complete embedded 
minimal disk in $V\subset \Omega$. Clearly $\Delta$ cannot be properly embedded 
in $\Omega$.  Nevertheless, the curvature estimates for stable minimal 
surfaces imply that $\overline{\Delta}$ is a smooth 
minimal lamination in $\Omega$. To determine the other leaves we note first 
that $\lim_{v\to \infty} 
\widehat{T}_v(\Delta)$ converges in ${\widehat \Omega}$ to some $\widehat{T}_v$ 
invariant minimal surface $\widehat{L}_+$ -- possibly, but not necessarily, 
$\widehat{\Gamma}_+$.  Similarly,  $\lim_{v\to - \infty} 
\widehat{T}_v(\widehat\Delta)$ converges 
to a $\widehat T_v$ invariant surface $\widehat L_-=\widehat{R}(\widehat L_+)$.  
As a consequence, $L_\pm =\widehat \Pi (\widehat L_\pm)$ are non-proper embedded 
minimal annuli that are leaves of $\overline \Delta$.  Finally, by construction 
$L_\pm$ are contained in $\overline{V}'$, that is are trapped between $\Gamma_+$ 
and $\Gamma_-$. Since the ends of $\Gamma_+$ and $\Gamma_-$ converge to the same 
side of $ \mathbb{T}^2_{1}$ and  $\mathbb{T}^2_{-1}$, one verifies that 
$\overline{L}_\pm = \mathbb{T}^2[{\pm 1}]$ and these are the remaining leaves 
of $\overline \Delta$. 
That is,
$\Delta=\widehat{\Pi}(\widehat{\Delta})$ is the desired minimal disk.

\section{Curvature Estimate of Schoen and Simon} \label{SScurvEstSec}
For the convenience of the reader we state the curvature estimate for embedded minimal disks  with a uniform area bound proved by  Schoen and Simon in \cite{Schoen1983a}.
\begin{thm}\label{SScurvEstThm}
 Fix $(\Omega,g)$ a Riemannian three-manifold and let 
$\mathcal{B}_{2r}(p)\subset N$ satisfy:
 \begin{itemize}
  \item $\exp^\Omega_p:B_{2r}(0)\to \mathcal{B}_{2r}(p)$ is a smooth 
diffeomorphism;
  \item With $g_{ij}dx^i dx^j=(\exp^\Omega_p)^* g$,  there is an $0<\alpha\leq 
1$ so
   \[
    \frac{1}{2} \delta_{ij}<g_{ij} <2 \delta_{ij}, \; \; \sup_{B_{2r}}\left(r \left| \frac{\partial g_{jk}}{\partial x^i} \right|+r^2 \left| \frac{\partial^2 g_{kl}}{\partial x^i \partial x^j} \right|\right)< 1,
    \]
    and 
    \[
     \sup_{(x,y)\in B_{2r}\times B_{2r}} r^{2+\alpha} |x-y|^{-\alpha} \left| \frac{\partial^2 g_{kl}}{\partial x^i \partial x^j}(x)-\frac{\partial^2 g_{kl}}{\partial x^i \partial x^j}(y)\right|< 1.
    \]
 \end{itemize}
    Given $\mu>0$, there is a $C=C(\mu)>0$ so that if $\Sigma\subset \mathcal{B}_{2r}(p)$ is a properly embedded minimal disk in $\mathcal{B}_{2r}(p)$ and $|\Sigma|\leq \mu r^2$, then 
  \[
   \sup_{\mathcal{B}_r(p)\cap \Sigma} |A|^2 < C r^{-2}.
  \]
\end{thm}


\providecommand{\bysame}{\leavevmode\hbox to3em{\hrulefill}\thinspace}
\providecommand{\MR}{\relax\ifhmode\unskip\space\fi MR }
\providecommand{\MRhref}[2]{%
  \href{http://www.ams.org/mathscinet-getitem?mr=#1}{#2}
}
\providecommand{\href}[2]{#2}

\end{document}